\documentclass[12pt]{article}
\usepackage{amsthm}
\usepackage{amsfonts}
\usepackage{epsfig}
\usepackage{amssymb}
\usepackage{amsmath}
\usepackage{amscd}
\usepackage{colonequals}
\usepackage{url}
\newtheorem{theorem}{Theorem}
\newtheorem{lemma}[theorem]{Lemma}
\newtheorem{problem}[theorem]{Problem}
\newtheorem{corollary}[theorem]{Corollary}

\newcommand{\Aa}{\mathcal{A}}

\title{Triangle-free graphs of tree-width $t$ are $\lceil (t+3)/2\rceil$-colorable}

\author{Zden\v{e}k Dvo\v{r}\'{a}k\thanks{Charles University, Prague, Czech Republic.
E-mail: {\tt rakdver@iuuk.mff.cuni.cz}.  Supported by project 17-04611S (Ramsey-like aspects of graph
coloring) of Czech Science Foundation.}\and
Ken-ichi Kawarabayashi\thanks{National Institute of Informatics,
2-1-2 Hitotsubashi, Chiyoda-ku, Tokyo 101-8430, Japan.
E-mail: {\tt k\_keniti@nii.ac.jp}. Supported by JST ERATO Grant Number JPMJER1305, Japan.}}
\date{}
\begin{document}
\maketitle

\begin{abstract}
We prove that every triangle-free graph of tree-width $t$ has chromatic number at most $\lceil (t+3)/2\rceil$,
and demonstrate that this bound is tight.  The argument also establishes a connection between
coloring graphs of tree-width $t$ and on-line coloring of graphs of path-width $t$.
\end{abstract}

While there exist triangle-free graphs of arbitrarily large chromatic number,
forbidding triangles improves bounds on the chromatic number in certain graph classes.
For example, planar graphs are $4$-colorable~\cite{AppHak1,AppHakKoc} and even deciding their $3$-colorability
is an NP-complete problem~\cite{garey1979computers}, while all triangle-free planar graphs are $3$-colorable~\cite{grotzsch1959}.
A graph on $n$ vertices may have chromatic number up to $n$, while Ajtai et al.~\cite{akomsem}
and Kim~\cite{kim1995ramsey} proved a tight upper bound $O\bigl(\sqrt{n/\log n}\bigr)$ on the chromatic number of triangle-free $n$-vertex graphs.
The chromatic number of graphs of maximum degree $\Delta$ may be as large as $\Delta+1$, while
Johansson~\cite{johansson1996asymptotic} proved that triangle-free graphs of maximum degree $\Delta$ have
chromatic number $O(\Delta/\log\Delta)$.

On the other hand, no such improvement is possible for graphs with bounded degeneracy.  A graph $G$ is \emph{$d$-degenerate} if each subgraph
of $G$ contains a vertex of degree at most $d$.  A straightforward greedy algorithm colors every $d$-degenerate
graph using $d+1$ colors; and a construction by Blanche Descartes~\cite{blades} gives for every positive integer $d$ a $d$-degenerate
triangle-free graph that is not $d$-colorable, as observed by Kostochka and Ne\v{s}et\v{r}il~\cite{konedes}.
A construction of triangle-free graphs with high chromatic number by Zykov~\cite{zykov} also has this property,
and the same construction was given (and possibly independently rediscovered) by Alon, Krivelevich, and Sudakov~\cite{triadegen}.

In this note, we consider the chromatic number of graphs of given tree-width $t$.
Note that every graph of tree-width $t$ is $t$-degenerate, but on the other hand,
the clique $K_{t+1}$ has tree-width $t$, establishing $t+1$ as the tight upper bound
on the chromatic number of graphs of tree-width $t$.
We show that the bound can be improved by a constant factor if triangles are forbidden.

\begin{theorem}\label{thm-main}
For any positive integer $t$, every triangle-free graph of tree-width at most $t$ has chromatic number at most $\lceil (t+3)/2\rceil$.
\end{theorem}

We also show that this result is tight.  Actually, we give the following stronger result on graphs with bounded clique number.
Let $g(t,2)=g(0,k)=1$ for all integers $t\ge 0$ and $k\ge 2$.  For $t\ge 1$ and $k\ge 3$, let us inductively define
$g(t,k)=\lceil (t+1)/2\rceil+g(\lfloor (t-1)/2\rfloor,k-1)$.
\begin{theorem}\label{thm-lower}
For all integers $t\ge 0$ and $k\ge 2$, there exists a $K_k$-free graph of tree-width at most $t$ with chromatic number at least $g(t,k)$.
\end{theorem}
Note that $g(t,3)=\lceil (t+3)/2\rceil$ for all $t\ge 1$ and that $g(t,k)>\bigl(1-\tfrac{1}{2^{k-2}}\bigr)t$ for all $t\ge 0$ and $k\ge 2$.
Theorem~\ref{thm-lower} motivates the following question, which we were not able to resolve.
\begin{problem}
For integers $k\ge 4$ and $t\ge k-1$, what is the maximum chromatic number of $K_k$-free graphs of tree-width at most $t$?
\end{problem}
Let us remark that in the list coloring setting, the question is much easier to settle---the complete bipartite graph $K_{t,t^t}$ has tree-width $t$,
but there exists an assignment of lists of size $t$ to its vertices from that it cannot be colored, showing that no improvement analogous to Theorem~\ref{thm-main}
is possible.

Another natural question concerns graphs with larger girth.  The construction of Blanche Descartes~\cite{blades} actually produces $d$-degenerate
graphs of girth six that are not $d$-colorable, and Kostochka and Ne{\v{s}}et{\v{r}}il~\cite{konedes} generalize this result to graphs of arbitrarily
large girth.
\begin{problem}
For integers $g\ge 4$ and $t\ge g-1$, what is the maximum chromatic number of graphs of tree-width at most $t$ and girth at least $g$?
\end{problem}

\section{Tree-width, path-width, and on-line coloring}

To prove Theorems~\ref{thm-main} and \ref{thm-lower}, it is convenient to establish a connection
between chromatic number of graphs with given tree-width $t$ and an on-line variant of the chromatic number for
graphs of path-width $t$.  In on-line coloring~\cite{gyarfas1988line}, the graph to be colored is revealed vertex by vertex
and a color has to be assigned to each revealed vertex immediately, with no knowledge regarding the rest of the graph.
We need a variation on this idea, revealing the vertices in the order given by a path decomposition of the graph.

Let us first recall some definitions.
A \emph{tree decomposition} of a graph $G$ is a pair $(T,\beta)$, where $T$ is a tree and $\beta$ is a function
assigning to each vertex of $T$ a set of vertices of $G$, such that for every $uv\in E(G)$ there exists $z\in V(T)$
with $\{u,v\}\subseteq\beta(z)$, and such that for every $v\in V(G)$, the set $\{z:v\in\beta(z)\}$ induces a non-empty
connected subtree in $T$.  The \emph{width} of the decomposition is the maximum of $1+|\beta(z)|$ over all $z\in V(T)$,
and the \emph{tree-width} of $G$ is the minimum possible width of its tree-decomposition.  A \emph{path decomposition}
is a tree decomposition $(T,\beta)$ where $T$ is a path, and \emph{path-width} of $G$ is the minimum possible width of its
path decomposition.

Suppose $(P,\beta)$ is a path decomposition of a graph $G$, where $P=z_0z_1\ldots z_n$.  We say that the path decomposition
is \emph{nice} if $\beta(z_0)=\emptyset$ and $|\beta(z_i)\setminus\beta(z_{i-1})|=1$ for $i=1,\dots, n$.
Observe that each path decomposition can be transformed into a nice one of the same width.
A nice path decomposition gives a natural way to produce the graph $G$ by adding one vertex at a time:
the construction starts with the null graph, and in the $i$-th step for $i=1,\ldots, n$, the unique vertex $v_i\in \beta(z_i)\setminus\beta(z_{i-1})$
is added to $G$ and joined to its neighbors in $\beta(z_i)$.  An \emph{on-line coloring algorithm} $\Aa$
at each step of this process assigns a color to the vertex $v_i$ distinct from the colors of the neighbors of $v_i$ in $\beta(z_i)$.
Note that the assigned color cannot be changed later, and that the algorithm does not in advance know the graph $G$,
only its part revealed till the current step of the process.
Let $\chi_{\Aa}(G,P,\beta)$ denote the number of colors $\Aa$ needs to color $G$ when $G$ is presented to the algorithm $\Aa$ vertex by vertex
according to the nice path decomposition $(P,\beta)$.

Let us remark that we misuse the term ``algorithm'' a bit, as we are not concerned with the question of efficiency or even
computability in any model of computation.  Formally,
an on-line coloring algorithm $\Aa$ is a function from the set of triples $(H,Q,\gamma)$, where
$H$ is a graph with a nice path decomposition $(Q,\gamma)$, to the integers, such that the following holds.
Consider any graph $G$ with a nice path decomposition $(P,\beta)$, where $P=z_0z_1\ldots z_n$ for some $n\ge 1$.
For $i=1,\ldots, n$, let $v_i$ be the unique vertex in $\beta(z_i)\setminus\beta(z_{i-1})$, let
$P_i=P-\{z_{i+1},\ldots,z_n\}$, $G_i=G-\{v_{i+1},\ldots, v_n\}$, and let $\beta_i$ be the restriction of $\beta$ to $V(P_i)$.
If $\varphi:V(G)\to\mathbf{Z}$ is defined by $\varphi(v_i)=\Aa(G_i,P_i,\beta_i)$, then $\varphi$ is a proper coloring of $G$.
And, $\chi_{\Aa}(G,P,\beta)$ is defined as the number of colors used by this coloring $\varphi$.

Nevertheless, it is easier to think about on-line coloring in the adversarial setting: an enemy is producing the graph
$G$ with its nice path decomposition on the fly and the algorithm $\Aa$ is assigning colors to the vertices as they arrive; and the enemy
can construct further parts of the graph depending on the coloring chosen by $\Aa$ so far.
We now present the main result of this section, which we use both to give upper bounds on the chromatic number of graphs with bounded
tree-width and to prove the existence of bounded tree-width graphs with large chromatic number.

\begin{lemma}\label{lemma-online}
Let $t$ and $k$ be positive integers, and let $c$ be the maximum of $\chi(G)$ over all $K_k$-free graphs $G$ of tree-width at most $t$.
There exists an on-line coloring algorithm $\Aa_t$ such that every $K_k$-free graph $H$ that has a nice path decomposition $(P,\beta)$ of width at most $t$
satisfies $\chi_{\Aa_t}(H,P,\beta)\le c$.

Conversely, if an on-line coloring algorithm $\Aa'_t$ satisfies
$\chi_{\Aa'_t}(H,P,\beta)\le c'$ for every $K_k$-free graph $H$ having a path decomposition $(P,\beta)$ of width at most $t$, then
$\chi(G)\le c'$ for every $K_k$-free graph $G$ of tree-width at most $t$.
\end{lemma}
\begin{proof}
Let us start with the second claim.  Let $G$ be a $K_k$-free graph and let $(T,\beta)$ be
its tree decomposition of width at most $t$.  Without loss of generality, we can assume that the tree $T$ is rooted in a vertex $r$
such that $\beta(r)=\emptyset$ and that each vertex $z\neq r$ of $T$ with parent $z'$ satisfies $|\beta(z)\setminus\beta(z')|=1$.
For each path $P$ in $T$ starting in $r$ and ending in a leaf of $T$, let $G_P$ be the subgraph of $G$ induced by $\bigcup_{z\in V(P)}\beta(z)$
and let $\beta_P$ be the restriction of $\beta$ to $V(P)$.  Then $(P,\beta_P)$ is a nice path decomposition of $G_P$ of
width at most $t$, and thus the algorithm $\Aa'_t$ can be used to color $G_P$ by at most $c'$ colors.  Furthermore, if $P'$ is another
root-leaf path in $T$ and some vertex $v$ belongs both to $G_P$ and $G_{P'}$, then $v\in \beta(z)$ for some vertex $z$ belonging to
the shared initial subpath of $P$ and $P'$, and thus the algorithm $\Aa'_t$ assigns the same color to $v$ in its run on $(G_P,P,\beta_P)$ and on $(G_{P'},P',\beta_{P'})$.
We conclude that the colorings of the graphs $G_P$ for all root-leaf paths $P$ are consistent, and their union
gives a proper coloring of $G$ using at most $c'$ colors.

Conversely, let $c$ be as given in the statement of the lemma.  Consider the ``universal'' infinite $K_k$-free graph $G$
of tree-width at most $t$.  That is, $G$ is an infinite graph with rooted tree decomposition $(T,\beta)$ of width at most $t$
satisfying the following conditions.
The root $r$ of $T$ has $\beta(r)=\emptyset$.  For every vertex $z$ of $T$ and all sets $I\subseteq B\subseteq \beta(z)$
such that $G[I]$ is $K_{k-1}$-free and $|B|\le t$, there exists a child $z'$ of $z$ in $T$
such that $\beta(z')=B\cup \{v\}$ for a vertex $v$ not belonging to $\beta(z)$, with $v$ adjacent in $G[\beta(z')]$
precisely to the vertices in $I$; and $z$ has no other children.  A standard compactness argument shows that
there exists a coloring $\varphi$ of $G$ using $c$ colors.  Now, for any $K_k$-free graph $H$ with a nice path decomposition $(P,\beta)$
of width at most $t$, we can identify $P$ with a path in $T$ and $H$ with an induced subgraph of $G$ in the natural way, and the algorithm
$\Aa_t$ works by assigning colors to vertices of $H$ according to the coloring $\varphi$.  Hence, $\chi_{\Aa_t}(H,P,\beta)\le c$.
\end{proof}

\section{Non-colorability}

In this section, we give the construction proving Theorem~\ref{thm-lower}.

Let $H_0$ and $H$ be graphs with nice path decompositions $(P_0,\beta_0)$ and $(P,\beta)$, respectively.
We say that $(H,P,\beta)$ \emph{extends} $(H_0,P_0,\beta_0)$ if $P_0$ is an initial segment of $P$,
$\beta_0$ is the restriction of $\beta$ to $V(P_0)$, and $H_0$ is an induced subgraph of $H$.
Let $z$ be the last vertex of $P$.  We say that a coloring $\varphi$ of $H$ is \emph{$c$-forced}
if there exists an independent set $I\subseteq \beta(z)$ of size $c$ such that vertices of $I$ receive pairwise
distinct colors according to $\varphi$.

\begin{lemma}\label{lemma-lbound-constr}
Let $k\ge 3$ and $c\ge 1$ be integers, and let $\Aa$ be an on-line coloring algorithm.
Let $H_0$ be a triangle-free graph with a nice path decomposition $(P_0,\beta_0)$ of width at most $2c-2$.
Let $c_0\le c-1$ be a non-negative integer.  If the coloring $\varphi_0$ of $H_0$ produced by $\Aa$ is $c_0$-forced,
then there exists a triangle-free graph $H$ with a nice path decomposition $(P,\beta)$ of width at most $2c-2$
such that $(H,P,\beta)$ extends $(H_0,P_0,\beta_0)$ and the coloring $\varphi$ of $H$ produced by $\Aa$ is $(c_0+1)$-forced.
\end{lemma}
\begin{proof}
Let $z_0$ be the last vertex of $P_0$ and let $I=\{v_1,\ldots, v_{c_0}\}$ be an independent set contained in $\beta_0(z_0)$
such that $\varphi_0$ assigns pairwise distinct colors to vertices of $I$.
Without loss of generality $\varphi_0(v_1)=1$, \ldots, $\varphi_0(v_{c_0})=c_0$.  As $(H,P,\beta)$ will be chosen
to extend $(H_0,P_0,\beta_0)$, the coloring $\varphi$ produced by the algorithm $\Aa$ will match $\varphi_0$ on $V(H_0)$.

We will now append further vertices $x_1$, $x_2$, \ldots at the end of $P_0$ to obtain the path $P$ as follows.
Let $\beta(x_1)=\{v_1, \ldots, v_{c_0},v'_1\}$ for a new vertex $v'_1$
with no neighbors, and use the algorithm $\Aa$ to extend $\varphi$ to this vertex.  If $\varphi(v'_1)$ is distinct from $1$, \ldots, $c_0$, then
we stop the construction.  Otherwise, we can by symmetry assume that $\varphi(v'_1)=1$.  We then let $\beta(x_2)=\{v_1,\ldots, v_{c_0},v'_1,v'_2\}$ for
a new vertex $v'_2$ adjacent to $v_1$.  If $\varphi(v'_2)$ is distinct from $1$, \ldots, $c_0$, then we stop the construction (the independent
set receiving $c_0+1$ distinct colors is $\{v'_1,v_2,\ldots,v_{c_0},v'_2\}$).  Otherwise, since $\varphi(v'_2)\neq\varphi(v_1)$, we can assume that $\varphi(v'_2)=2$.
Similarly, we proceed for $j=3,\ldots, c_0+1$: we set $\beta(x_j)=\{v_1,\ldots,v_{c_0},v'_1,\ldots, v'_j\}$, with $v'_j$ adjacent to $v_1$, \ldots, $v_{j-1}$,
and depending on the decision of $\Aa$ regarding the color of $v'_j$, we either stop with the independent
set $\{v'_1,v'_2,\ldots, v'_{j-1},v_j,v_{j-1},\ldots, v_{c_0},v'_j\}$ using $c_0+1$ distinct colors, or we can assume
that $\varphi(v'_j)=j$.  The former happens necessarily at latest when $j=c_0+1$.

All the bags created throughout this process have size at most $2c_0+1\le 2c-1$, and thus the width of the resulting path decomposition
is at most $2c-2$.
\end{proof}

Iterating Lemma~\ref{lemma-lbound-constr}, we obtain the following.

\begin{corollary}\label{cor-lbound}
Let $t\ge 0$ be an integer. For any on-line coloring algorithm $\Aa$,
there exists a triangle-free graph $H$ with a nice path decomposition $(P,\beta)$ of width at most $t$
such that the coloring $\varphi$ of $H$ produced by $\Aa$ is $\lceil\tfrac{t+1}{2}\rceil$-forced.
\end{corollary}

We are now ready to establish Theorem~\ref{thm-lower}.

\begin{proof}[Proof of Theorem~\ref{thm-lower}]
By Lemma~\ref{lemma-online}, it suffices to show that for every on-line coloring algorithm $\Aa$,
there exists a $K_k$-free graph $H$ with a nice path decomposition $(P,\beta)$ of width at most $t$ such that $\chi_{\Aa}(H,P,\beta)\ge g(t,k)$.
We prove the claim by induction on $k$.  When $k=2$ or $t=0$, we have $g(t,k)=1$ and the claim obviously holds, with $H$ being a single-vertex graph.
Hence, we can assume that $k\ge 3$ and $t\ge 1$.
Let $c_1=\lceil\tfrac{t+1}{2}\rceil$ and $c_2=t-c_1=\lfloor\tfrac{t-1}{2}\rfloor$.

Let $H_0$ be a triangle-free graph with a nice path decomposition $(P_0,\beta_0)$ of width at most $t$
such that the coloring $\varphi_0$ of $H_0$ produced by $\Aa$ is $c_1$-forced, obtained using Corollary~\ref{cor-lbound}.  Let $z_0$ be the last vertex of $P_0$
and let $I$ be an independent set in $H_0[\beta_0(z_0)]$ of size $c_1$ whose vertices are colored by pairwise distinct colors in $\varphi_0$.

Let $\Aa'$ be an on-line coloring algorithm defined as follows: given any graph $H_1$ with a nice path decomposition $(P_1,\beta_1)$,
let $H'_1$ be the graph obtained from a disjoint union of $H_0$ and $H_1$ by adding all edges between $I$ and $V(H_1)$.
Let $P'_1$ be the concatenation of $P_0$ and $P_1$, with $\beta'_1(z)=\beta_0(z)$ for $z\in V(P_0)$ and $\beta'_1(z)=\beta_1(z)\cup I$
for $z\in V(P_1)$.  Then $(P'_1,\beta_1)$ is a nice path decomposition of $H'_1$.  The algorithm $\Aa'$ obtains a coloring of $H_1$ as
the restriction of the coloring of $H'_1$ given by the algorithm $\Aa$ to $V(H_1)$.

By the induction hypothesis, there exists a $K_{k-1}$-free graph $H_1$ with a nice path decomposition $(P_1,\beta_1)$ of width at most $t-c_1=c_2$
such that $\chi_{\Aa'}(H_1,P_1,\beta_1)\ge g(c_2,k-1)$.  Let $H$ and its path decomposition $(P,\beta)$ be chosen as the corresponding
graph $H'_1$ with path decomposition defined as in the previous paragraph.
Since $H_1$ is $K_{k-1}$-free, $H_0$ is triangle-free, and $I$ is an independent set, we conclude that $H$ is $K_k$-free.
Let $\varphi$ be the coloring of $H$ obtained by $\Aa$.  The restriction of $\varphi$ to $V(H_0)$ matches $\varphi_0$,
and in particular $c_1$ distinct colors are used on $I$.  The restriction of $\varphi$ to $V(H_1)$ by definition
matches the coloring of $H_1$ obtained by the algorithm $\Aa'$, and thus it uses at least $g(c_2,k-1)$ distinct colors.
Furthermore, since each vertex of $I$ is adjacent to all vertices of $V(H_1)$ in $H$, the sets of colors used on $I$ and on $V(H_1)$ are disjoint.
Therefore, $\chi_{\Aa}(H,P,\beta)\ge c_1+g(c_2,k-1)=g(t,k)$.
\end{proof}

\section{Colorability}

Let $c'$ be a positive integer, let $F$ be a graph, and let $\varphi$ be a $c'$-coloring
of $F$.  We say that $\varphi$ is \emph{$F$-valid} if $F$ does not contain any
independent set on which $\varphi$ uses all $c'$ distinct colors.
We say that a color $a$ is \emph{$(F,\varphi)$-forbidden} if there exists an independent set
$A_a\subseteq V(F)$ in $F$ such that $\varphi$ uses all colors except for $a$ on $A_a$.  We need the following auxiliary claim.

\begin{lemma}\label{lemma-numforb}
Let $c'$ be a positive integer and let $\varphi$ be a $c'$-coloring of a graph $F$.
If $\varphi$ is $F$-valid, then at most $\max(|V(F)|-c'+2,0)$ colors are $(F,\varphi)$-forbidden.
\end{lemma}
\begin{proof}
We prove the claim by induction on $|V(F)|$, and thus we assume that Lemma~\ref{lemma-numforb} holds for all
graphs with fewer than $|V(F)|$ vertices.  For each $(F,\varphi)$-forbidden color $a$, let $A_a$ be an independent set
such that $\varphi$ uses all colors except for $a$ on $A_a$.

If $|V(F)|\le c'-2$, then $F$ contains no independent set of size $c'-1$, and thus no color is $(F,\varphi)$-forbidden.
Hence, suppose that $|V(F)|\ge c'-1$, and thus $|V(F)|-c'+2\ge 1$.  If at most one color is $(F,\varphi)$-forbidden, then
the lemma holds.  Hence, we can by symmetry assume that colors $1$ and $2$ are forbidden.  Note that all $c'$ colors
appear at least once on $A_1\cup A_2$.  If each $(F,\varphi)$-forbidden color is used on at least two vertices of $F$,
then the number of $(F,\varphi)$-forbidden colors is at most $|V(F)|-c'$, and the lemma holds.

Hence, we can assume that the color $c'$ is $(F,\varphi)$-forbidden and used on exactly one vertex $v$ of $F$.
Let $F'$ be the graph obtained from $F$ by removing $v$ and all the neighbors of $v$, and let $\varphi'$ be the restriction
of $\varphi$ to $F'$.  Note that $\varphi'$ is a $(c'-1)$-coloring of $F'$.  We claim that $\varphi'$ is $F'$-valid.
Indeed, if all colors $1$, \ldots, $c'-1$ were used on an independent set $A'\subseteq V(F')$, then all colors $1$, \ldots, $c$
would be used on the independent set $A'\cup\{v\}$ in $F$, contradicting the assumption that $\varphi$ is $F$-valid.
If a color $a\neq c'$ is $(F,\varphi)$-forbidden, then note that $A_a$ contains $v$ since $v$ is the only vertex of $F$ of color $c'$,
and does not contain any of the neighbors of $v$ since $A_a$ is an independent set.  Hence, $A_a\setminus\{v\}$ is an independent
set in $F'$ on that all colors except for $a$ appear, and thus $a$ is $(F',\varphi')$-forbidden.  Denoting by $f'$ the number of
$(F',\varphi')$-forbidden colors, we conclude that at most $f'+1$ colors are $(F,\varphi)$-forbidden.

Since $\varphi$ is $F$-valid, the set $A_{c'}\cup\{v\}$ is not independent, and thus $v$ has degree at least one.
Consequently, $|V(F')|\le |V(F)|-2$.  By the induction hypothesis, we conclude that the number of $(F,\varphi)$-forbidden colors
is at most
$$\max(|V(F')|-(c'-1)+2,0)+1\le\max(|V(F)|-c'+2,1)=|V(F)|-c'+2,$$
as required.
\end{proof}

We are now ready to bound the chromatic number of triangle-free graphs of tree-width at most $t$.

\begin{proof}[Proof of Theorem~\ref{thm-main}]
We need to show that every triangle-free graph $G$ of tree-width at most $t$ can be colored using $c'=\lceil (t+3)/2\rceil$
colors.  Note that $2c'-2>t$.  By Lemma~\ref{lemma-online}, it suffices to design an on-line coloring algorithm $\Aa'$
that colors every triangle-free graph $H$ with a nice path decomposition $(P,\beta)$ of width at most $t$ using at most $c'$ colors.

The algorithm $\Aa'$ maintains the invariant that the restriction of the $c'$-coloring $\varphi$ produced by this algorithm to $\beta(z)$
is $H[\beta(z)]$-valid for every $z\in V(P)$.  Consider any vertex $z'$ with predecessor $z''$ in $P$, let
$v$ be the unique vertex in $\beta(z')\setminus\beta(z'')$, and let $N$ be the set of neighbors of $v$ in $\beta(z')$.
Since $H$ is triangle-free, $N$ is an independent set, and thus at most $\min(|N|,c'-1)$ colors appear on $N$ by the invariant.
To get a proper coloring maintaining the invariant, it suffices to assign $v$ an arbitrary color that does not appear on $N$
and that is not $(H[\beta(z')\setminus(N\cup\{v\})],\varphi)$-forbidden.
Since $|\beta(z')|\le t+1$, Lemma~\ref{lemma-numforb} implies that the number of
such colors is at least
$$c'-\min(|N|,c'-1)-\max(t-|N|-c'+2,0).$$
If $|N|\ge c'$, this is at least $c'-(c'-1)-\max(t-2c'+2,0)=1$, using the fact that $2c'-2>t$.
If $|N|\le c'-1$, this is at least $c'-|N|-\max(t-|N|-c'+2,0)=\min(2c'-t-2,c'-|N|)\ge 1$.
In either case it is possible to extend the coloring.
\end{proof}

\section*{Acknowledgments}

We would like to thank Endre Cs\'oka, Tom\'a\v{s} Kaiser, and Edita Rollov\'{a} for fruitful discussions which
led to a simplification of the proof of Lemma~\ref{lemma-numforb}.

\bibliographystyle{siam}
\bibliography{trfreetw}

\end{document}